\documentclass[11pt]{amsart}

\usepackage{graphicx,colordvi,amssymb,colortbl,color,epsfig}
\usepackage{enumerate}
\usepackage[top=1in, bottom=1in, left=1in, right=1in]{geometry}
\usepackage{algorithm}
\usepackage{algpseudocode}

\newcommand{\dia}{$\diamond$}

\newcommand{\Z}{\mathbb{Z}}

\newcommand{\Q}{\mathbb{Q}}

\newcommand{\F}{\mathbb{F}}

\newlength{\khov}
\newlength{\burg}
\newlength{\jmr}
\newlength{\smale}
\newtheorem{dfn}{Definition}[section]
\newtheorem{lemma}[dfn]{Lemma}
\newtheorem{cor}[dfn]{Corollary}

\newtheorem{ex}[dfn]{Example}
\newtheorem{thm}[dfn]{Theorem}

\newtheorem{prop}[dfn]{Proposition}

\renewcommand{\qed}{$\blacksquare$}

\newcommand{\rmv}[1]{}
\author{Qi Cheng} 
\email{qcheng@ou.edu}
\address{School of Computer Science, 
University of Oklahoma, Norman, OK \ 73019} 
\author{Shuhong Gao} 
\email{sgao@math.clemson.edu} 
\address{Department of Mathematical Sciences, 
Clemson University, Clemson, SC \ 29634-0975} 
\author{J.\ Maurice Rojas}
\email{rojas@math.tamu.edu} 
\address{TAMU 3368, College Station, TX \ 77843-3368} 
\thanks{Partially supported by NSF grant CCF-1409020, 
the American Institute of Mathematics, 
and MSRI (through REU grant DMS-1659138). }  
\author{Daqing Wan} 
\email{dwan@math.uci.edu} 
\address{Department of Mathematics, University of California, 
Irvine, CA\  92697-3875 }

\title{Counting Roots of Polynomials Over Prime Power Rings}
\begin{document}
\pagenumbering{gobble} 

\begin{abstract} 
Suppose $p$ is a prime, $t$ is a positive integer, and 
$f\!\in\!\Z[x]$ is a univariate polynomial of degree $d$ with coefficients 
of absolute value $<\!p^t$. We show that for any {\em fixed} $t$, we can 
compute the number of roots in $\Z/(p^t)$ of $f$ in deterministic 
time $(d+\log p)^{O(1)}$. This fixed parameter tractability appears 
to be new for $t\!\geq\!3$. A consequence for arithmetic geometry is that 
we can efficiently compute Igusa zeta functions $Z$, for univariate 
polynomials, assuming the degree of $Z$ is fixed.   
\end{abstract} 

\maketitle

\newpage 

\thispagestyle{plain} 
\pagenumbering{arabic} 
\setcounter{page}{1} 

\section{Introduction}
Given a prime $p$, and a polynomial $f\in \Z [x]$ of degree $d$ with 
coefficients of absolute value $<p^t$, it is
a basic problem to count the roots of $f$ in $ \Z/(p^t) $. Aside from 
its natural cryptological relevance, counting roots in $\Z/(p^t)$ 
is closely related to factoring polynomials over the $p$-adic 
rationals $\Q_p$ \cite{chistov,cantorqp,gnp}, and the latter problem is 
fundamental in polynomial-time factoring over the rationals \cite{lll}, 
the study of prime ideals in number fields \cite[Ch.\ 4 \& 6]{cohenant}, 
elliptic curve cryptography \cite{lauder},  the computation of 
zeta functions \cite{denefver,lauderwan, dwan, chambert}, and 
the detection of rational points on curves \cite{bjornbm}. 
 
There is surprisingly little written about root counting in 
$\Z/(p^t)$ for $ t\geq 2 $: While an algorithm for counting roots of 
$f$ in $\Z/(p^t)$ in time polynomial in $d+\log p$ has been 
known in the case $t=1$ for many decades (just compute the degree of 
$\gcd(x^p-x,f)$ in $\F_p[x]$), the case $t=2$ was just solved in 2017 by 
some of our students \cite{hjpw}. The cases $t\!\geq\!3$, which we solve 
here, appeared to be completely open. One complication with $t\geq 2$ 
is that polynomials in $(\Z/(p^t))[x]$ do not have unique factorization, 
thus obstructing a simple use of polynomial gcd. 

However, certain basic facts can be established quickly. 
For instance, the number of roots can be exponential in $\log p$. 
(It is natural to use $\log p$, among other parameters, to measure the size 
of a polynomial since it take $O(t\log p)$ bits to specify a solution in 
$\Z/(p^t)$.) The quadratic polynomial $ x^2 = 0 $, which has roots 
$ 0, p, 2p, \cdots, (p-1)p $ in $ \Z/(p^2)$, is such an example. 
This is why we focus on computing the number of roots of $f$, instead of 
listing or searching for the roots in $\Z/(p^t)$. 

Let $N_t(f)$ denote the number of roots of $f$
in $\Z/(p^t)$ (setting $N_0(f)\!:=\!1$). The
{\em Poincare series for $f$} is $P(x)\!:=\!\sum^\infty_{t=0} N_t(f)x^t$.
Assuming $P(x)$ is a rational function in $x$, one can reasonably recover
$N_t(f)$ for any $t$ via standard generating function techniques. That $P(x)$ 
is in fact a rational function in $x$ was first proved in 1974 by Igusa (in
the course of deriving a new class of zeta functions \cite{igusa}),
applying resolution of singularities. Denef found a new proof 
(using $p$-adic cell decomposition \cite{denef}) leading to more
algorithmic approaches later. While this in principle gives us a 
way to compute $N_t(f)$, there are few papers studying the computational 
complexity of Igusa zeta functions \cite{zg}. Our work here thus also 
contributes in the direction of arithmetic geometry by significantly improving 
\cite{zg}, where it is assumed  that $f(x)$ splits completely over $\Q$. 

To better describe our results, let us start with a naive description 
of the first key idea: How do roots in $ \Z/(p) $ lift to roots in 
$ \Z/(p^t) $? A simple root of $f$ in $ \Z/(p) $ can be lifted uniquely to
a root in $ \Z/(p^t) $, according to the classical Hensel's lemma 
(see, e.g., \cite{gouvea}). But a root with multiplicity $\geq 2$ in $ \Z/(p) $ 
can potentially be the image (under mod $p$ reduction) of many 
roots in $\Z/(p^t)$, as illustrated by our earlier example $f(x)\!=\!x^2$. 
Or a root may not be liftable at all, e.g., $ x^2 + p = 0 $ has no roots mod 
$p^2$, even though it has a root mod $p$. More to the point, if one wants a 
fast deterministic algorithm, one can not assume that 
one has access to individual roots. This is because  it is still an open problem whether there exists 
a deterministic polynomial time algorithm for finding roots of polynomials
 modulo $p$, see for example \cite{Gao01,kedlaya}.

Nevertheless, we have overcome this difficulty and found a way to keep track 
of how to correctly lift roots of any multiplicity. 
\begin{thm}\label{main}
There is a deterministic algorithm that computes the number, 
$N_t(f)$, of roots in $\Z/(p^t)$ of $f$ in time $(d+\log(p)+2^t)^{O(1)}$.  
\end{thm}

\noindent 
Theorem \ref{main} is proved in Section \ref{sec:gen}. 
Note that Theorem \ref{main} implies that if $ t= O(\log\log p) $ then there 
is a deterministic $(d+\log p)^{O(1)}$ algorithm to count the roots of $f$ 
in $ \Z/(p^t) $.  We are unaware of any earlier algorithm achieving this 
complexity bound, even if randomness is allowed.

Our main technical innovations are the following: 
\begin{itemize}
\item We use ideals in the ring $\Z_p[x_1,\ldots,x_k]$ of multivariate 
polynomials over the $p$-adic integers to keep 
track of the roots of $f$ in $\Z/(p^t)$. 
  More precisely, from the expansion:
  \[ f(x_1+ px_2 + \cdots + p^k x_{k-1}) = g_1 (x_1) + p g_2(x_1, x_2) + p^2
    g_3 (x_1, x_2, x_3) + \cdots  \]
  we build a collection of ideals in $\Z_p[x_1,\ldots,x_k]$, starting from 
$ (g_1 (x_1)) $. 
  We can then decompose the ideals  
  according to multiplicity type and rationality. This process
  produces a tree of ideals which will ultimately encode the 
summands making up our final count of roots. 
\item The expansion above is not unique. (For example, adding $ p $ to $ g_1 $
  and subtracting $ 1 $ from $ g_2 $  gives us another expansion.)  
  However, we manage to keep most of our computation within $ \Z/(p) $, 
  and maintain uniformity for the roots of our intermediate ideals,  
  by using Teichmuller lifting (described in Section \ref{sec:firstteich}).  
\end{itemize}

\section{Overview of Our Approach}
To count the number of roots in $ \Z/(p^t)$ of $f\in \Z[x] $, 
our algorithm follows a divide-and-conquer strategy. First,  
factor $ f $ over $ \F_p $ as follows:   
\begin{equation}\label{modpfac}
 f(x) = f_1 (x) f_2^2 (x) f_3^3 (x) ... f_l^l(x) g(x)    \pmod{ p}, 
\end{equation}
where each $f_i$ is a monic polynomial over $ \F_p $ that can be split into a 
product of distinct linear factors over $\F_p$, and the $f_i$ are pairwise
relatively prime, and $ g(x) $ is free of linear factors in $ \F_p[x] $.
For an element $ \alpha\in \F_p $, we call its pre-image
under the natural map $ \Z \rightarrow \F_p $
a lift of $ \alpha $ to $ \Z $. Similarly, we can define a lift
of $ \alpha $ to $ \Z_p $ (the $p$-adic integers) or to $ \Z/(p^t) $.
We extend the concept to polynomials in $ \F_p [x] $.
The core of our algorithm counts 
how many roots of $ f $ in $ \Z/(p^t) $ are lifts of roots of $ f_i $ in 
$ \F_p $, for each $ i $ .
For $ f_1 $,  by Hensel's lifting
lemma, the answer should be $ \deg f_1 $ for all $ t $.
For other $ f_i $, however, 
Hensel's lemma will not apply, so we run our 
algorithm on the pair $(f,m)$, where $ m $ is the lift
of $ f_i $ to $ \Z [x] $,   for each $i\in\{2,\ldots,l\} $, 
to see how many lifts (to roots of $f$ in $\Z/(p^t)$) 
are produced by the roots of $f_i$ in $\Z/(p)$.
The final count will be the summation of the results over all the $ f_i $, 
since the roots of $f$ in $ \Z/(p^t) $  are partitioned by the roots of 
the $f_i$.

The first step of the algorithm (when applied to a pair $(f,m)$) 
is to find the maximum positive integer
$ s $ such that there exists a polynomial 
such that 
\[ f(x_1 + p x_2) = p^s g(x_1, x_2)  \pmod{
    (m(x_1), p^t)}.\]
We may assume that
\[ g(x_1, x_2) = \sum_{0\leq j < t} g_j(x_1) x_2^j,  \]
and for all $ j $, either $g_j=0$ or $\gcd(m(x_1), g_j(x_1)) =1$ over $ \F_p $. 
(Otherwise some $f_i$ can be split further, and we
restart the algorithm with new $m$'s of smaller degrees.) 
Since $m| f$ over  $ \F_p[x] $, such $ s $ and $ g $ exist,
and can be found efficiently. 

If $ s \geq t $, then each root of $ m $  in
$ \F_p $ lifts to $ p^{t-1} $ roots of $ f $ in $ \Z/(p^t) $.

If $ s< t$,
let $r\in\F_p$ be any root of $m$ and let $r'$ be the corresponding lifted 
root of $ m $ in $ \Z_p$.
We then have
\[  f(r' + ap) = p^s g(r',a)  \pmod{ p^t}. \]
So $r' + ap$ is a root in $ \Z/(p^t)$ for $f$ if and only if 
  \[  g(r',a) = 0 \pmod{p^{t-s}}. \] 

 The preceding argument leads us to the following result. 
  \begin{prop}
    The number of roots in $\Z/(p^t)$ of $f$ 
    that are lifts of the roots of $m \pmod{p} $ 
    is equal to $p^{s-1}$ times the 
    number of solutions in $(\Z/(p^{t-s}))^2$ of the $2\times 2$ polynomial 
    system (in the variables $(x_1, x_2)$) below:
    \begin{equation}
\begin{split}
m(x_1)&=0 \\
g(x_1, x_2)&=0
\end{split}  \label{eqnpt}
    \end{equation}
  \end{prop}

  There is a dichotomy corollary from the above proposition.

  \begin{cor}
    If $ m^2 | f $ in $ \F_p[x] $, and $ t \geq 2 $ , then any root of $ m 
    $ in $ \F_p $ is either not liftable to a root in $ \Z/(p^t) $ of 
$ f $, or can be lifted to at
    least $ p $ roots of $ f $ in $ \Z/(p^t) $.
  \end{cor}

\subsection{The algorithm for $ t=3 $ }
Recall that our algorithm begins by seeking the maximal positive integer
$ s $ such that there is a polynomial $ g $ satisfying  
\[ f(x_1 + p x_2) = p^s g(x_1, x_2)  \pmod{
    (m(x_1), p^t)},\]
where $ m(x)\in \Z [x] $ and $ m^2 | f $ over $ \F_p $.  
If $s = 1$ then we must have 
\[ f = g' m^2 + p g   \]
for some polynomials $ g' $ and $ g $ with 
$ \gcd(m,g)=1 $ over $ \F_p $. 
None of the roots of $m$ in $ \F_p $ can then be lifted. 

If $s = 2$ then we have  $ f(x_1+px_2) = p^2 g(x_1, x_2) \pmod{m(x_1), p^3}$.  
\begin{cor}
The number of roots in $ \Z/(p^3) $ of $ f $ that
are lifts of roots of $ m\pmod{p} $ is equal to
$p$ times the number of roots in $\F^2_p$ of the $2\times 2$ 
polynomial system  below:
\begin{equation}
\begin{split}
m(x_1)&=0 \\
     g(x_1, x_2)&=0
\end{split} \label{eqnp3}
\end{equation}
which can be calculated in deterministic polynomial time. 
\end{cor}

Note that since the degree of $ x_2 $ in $ g $ is
at most $ 2 $ , any root of $m$ in
$\Z/(p)$ can be lifted to at most $2 p$ roots in $ \Z/(p^3)$.

If $ g(x_1, x_2)  $ is linear in $ x_2 $ , then counting points for 
(\ref{eqnp3}) is easy. The following theorem covers the other case.

\begin{thm}
  Assume that $ g(x_1, x_2) = x_2^2 + g'(x_1, x_2)  $,
  where the degree of $ x_2 $ in   $ g' $ is less than  $ 2 $.
Let $ M $  be the companion matrix of $ m$. 
Let $ X(x_1) $ be the discriminant  of the second
equation in (\ref{eqnp3}), viewed as a polynomial
in $x_2$.  
Let $ D(x_2) $ be the determinant of the matrix $ g(M, x_2) $. 
  Let $ E $ be the number of solutions of $ D(x_2) $ over $ \F_p $, counting
  multiplicity.
  The number of solutions of the $2\times 2$ polynomial system (\ref{eqnp3})  
is equal to $ E - \deg \gcd(m, X).$ 
\end{thm}

\begin{proof}
  Suppose that over $ \F_p $ \[ m(x) = \prod_{i=1}^c (x-\alpha_i).  \]
  Then $ M = V \mathrm{diag}(\alpha_1, \alpha_2, \cdots, \alpha_c) V^{-1} $ for some invertible matrix $
  V\in \F_p^{c \times c } $,
  where $\mathrm{diag}$ maps a vector into a diagonal matrix in the obvious 
way.  So $g(M, x_2)= V \mathrm{diag}(g(\alpha_i, x_2), \ldots ) V^{-1}$. 
  We then have
  \[ D(x_2) = \prod_{i=1}^c g(\alpha_i, x_2).   \]
  If $ a_1 $ is a solution of $ D(x_2) $ then there must exist a root
  $ \alpha_i $
  of $ m $ such that $ a_1 $ is a solution of 
    \[  g(\alpha_i, a_1) = 0.  \]
  
  If $ m | X $, then for every root $ \alpha_i $ of $ m $, 
  the above equation 
  has a solution in $ \F_p $ with multiplicity two, so the number of solutions
  of (\ref{eqnp3}) is $ E/2 (= c) $.

  If $ \gcd(m, X) = 1 $,  the equation has two distinct roots (which may not
  lie in $ \F_p$), and 
  the total number solutions of (\ref{eqnp3}) is $ E $. 
\end{proof}

Assume that $f\in\Z[x]$ is not divisible by $p$. 
The preceding ideas are formalized in the following algorithm: 
\begin{algorithm}\label{algot3}
  \caption{The case $ t=3 $ }
\begin{algorithmic}[1]
  \Function{count}{$f(x)\in \Z[x]$, $ f(x) \not= 0 \pmod{p} $ }
  \State Factor
\[ f(x) = f_1 (x) f_2^2 (x) f_3^3 (x) ... f_n^n(x) g(x)    \pmod{ p}. \]
  \State $ count = \deg f_1 $ 
  \Comment{Every roots of $ f_1 $ can be lifted uniquely. }
  \State  Push $  f_2(x), f_3(x), \cdots, f_s (x)  $ into a stack 
  \While{ $ S \not= \emptyset $} 
  \State Pop a polynomial from the stack, find its lift to $ \Z $  and denote it by $ m(x) $ 
  \State Find the maximum $ s $ and a polynomial $ g(x_1, x_2) $ such that
\[ f(x_1 + p x_2) = p^s g(x_1, x_2)  \pmod{
    (m(x_1), p^t)}.\]
  \If{ $ s=3 $ }
  \State $count  \leftarrow count + p^2 \deg m $ 
  \Else \If{$ s=2 $ }
  \State Let $ g'(x_1) $ be the leading coefficient of $ g(x_1, x_2) $, viewed
  as a polynomial in $ x_2 $.   
  \If{ $ \gcd(m, g') $ in $ \F_p[x] $  is nontrivial }
  \State Find the nontrivial factorization $ m(x) = m_1 (x) m_2 (x) $ in $ \F_p[x] $  
  \State Push $ m_1 $  and $ m_2 $ into the stack
  \Else
  \State $ count \leftarrow $ the number of
  the $ \F_p $-points of (\ref{eqnp3})   
    \EndIf
    \EndIf
    \EndIf
  \EndWhile
  \State \Return count
  \EndFunction
\end{algorithmic}
\end{algorithm}

\section{From Taylor Series to Ideals}
\label{sec:tay} 
For any polynomial    $m(x)$ of degree $n$, define
\[ T_{m,j} (x,y) = \sum_{1\leq i \leq j} \frac{y^{i-1}}{i! } \frac{d^i m}{(dx)^i} (x).  \]
Note that if $ m\in \Z[x]$ 
then $  \frac{1}{i! } \frac{d^i m}{(dx)^i} (x) $,  
being a Taylor expansion coefficient, also lies in $\Z[x]$. 
So $ T_{m, j} $ is an integral multivariate polynomial for any $ j $.
Since $ T_{m,1} $ does not depend on $ y $,  we abbreviate $ T_{m,1} (x,y) $
by $ T_m(x) $.
The following lemma follows from a simple application of Taylor expansion:
\begin{lemma}
Let $m\in\Z[x]$ be an irreducible polynomial that 
splits completely, without repeated factors, into linear factors over 
$\F_p$. Let $r\in\F_p$ be any root of $m$ and let $r'\in\Z_p$ 
be the corresponding $p$-adic integer root of $m$. Then
   \[ m(r'+a p) = a p T_m (r)  \pmod{ p^2}.\]
   To put it in another way, we have the following congruence: 
   \[ m(x_1 + p x_2 ) \equiv p x_2  T_m (x_1)  \pmod{ m(x_1), p^2} \]
   in the ring $ \Z[x_1, x_2]. $ 
\end{lemma}

That one can always associate an $r\in \Z/(p)$ to a root
             $r'\in\Z_p$ as above is an immediate consequence of the
             classical Hensel's Lemma \cite{gouvea}.
More generally, we have the following stronger result:  

\begin{lemma} \label{taylorgen}
Let $m\in\Z[x]$ be an irreducible polynomial that 
splits completely, without repeated factors, into linear factors over 
$\F_p$. Let $r\in\F_p$ be any root of $m$, 
and let $r'\in\Z_p$ be the corresponding $p$-adic integer root 
of $m$. Then for any positive integer $ u $, 
   \[ m(r'+a p) = a p T_{m, u-1}(r', a p)  \pmod{ p^u}. \]
   And in the ring $ \Z [x_1, x_2] $, we have
   \[  m(x_1 + p x_2 ) = x_2 p T_{m, u-1}(x_1,  p x_2 )  \pmod{ m(x_1), p^u}. \]
\end{lemma}

\begin{proof}
  By Taylor expansion:
  \begin{align*}
    m(r'+ ap) &= m(r') + \sum_{1\leq i < u} \frac{(a p)^i}{i! } \frac{d^i m}{(dx)^i} (r') \pmod{p^u}\\
    &= \sum_{1\leq i < u} \frac{(a p)^i}{i! } \frac{d^i m}{(dx)^i} (r')\pmod{p^u}\\
    &= a p \sum_{1\leq i < u} \frac{(a p)^{i-1}}{i! } \frac{d^i m}{(dx)^i} (r')\pmod{p^u}
  \end{align*}
  As observed earlier, $\frac{1}{i!} \frac{d^i m}{(dx)^i} (x) $ 
  is an integral polynomial (even when $ i > p-1 $), so we are done. 
\end{proof}

Note that  in the setting of Lemma~\ref{taylorgen}, $ T_{m, u-1}(r', a p) \equiv T_m (r') \not= 0
\pmod{p}. $ 

The following theorem is a generalization of the preceding lemmas to ideals.
\begin{thm}\label{Tayg}
  Let $ I $ be a ideal   in $ \Z_p [x_1, x_2, \cdots, x_{k-1}] $. 
  Assume that $ I \pmod{p} $ is a zero dimensional radical ideal in
  $ \F_p [x_1, x_2, \cdots, x_{k-1}] $
  with only rational roots.
  Let $ f(x_1, x_2, \cdots, x_k) $ be an integer polynomial whose degree on $
  x_k $  is less than $ p $. If
  $ f(r_1, r_2, \cdots, r_k) \equiv 0 \pmod{p^s} $ for
  every $ \Z_p $-root $ (r_1, r_2, \cdots, r_{k-1}) $ of $ I $,
  and every integer $ r_k $,
  then there must exist a polynomial $ g (x_1, x_2, \cdots, x_k) $ 
  such that \[ f(x_1, x_2, \cdots, x_k)
    \equiv p^s g (x_1, x_2, \cdots, x_k) \pmod{I}.  \]
\end{thm}

The theorem can be proved by induction on $ s $. 
Lemma~\ref{taylorgen} is basically the special case of Theorem~\ref{Tayg}
when $ s=1 , k =2 $, $ I = (m (x_1)) $ and $ f(x_1, x_2) = m(x_1+px_2) $.  
It is important that the ideal $ I \pmod{p} $ need to be radical, 
just like in Lemma~\ref{taylorgen}, $ m(x) $ need to be free
of repeated factors over $ \F_p $. 

\section{The Case $ t=4 $ and the Need for Teichmuller Lifting. }
\label{sec:firstteich} 
Here we work on the case $ t=4 $.
Earlier, we saw that $ m(x) $ can be
taken to be any lift of $ f_i $  to $ \Z [x] $.
In this section we will use Teichmuller lifting to
get some uniformity needed by our algorithm.
We start with
\[ f (x_1 + px_2) = p^s g(x_1, x_2) \pmod{m(x_1),p^4}. \]
The simplest subcase is $s=4$. 
Every root of $ m(x) $ can be lift to $ p^3 $ many roots of $ f $   
in $ \Z/p^4 $.

If $s =3 $,
we have

\begin{thm}
  The number of roots in $ \Z/(p^4) $ of $ f $ that
  are lifts of roots of $ m\pmod{p} $ is equal to
$p^2 $ times the number of roots in $\F^2_p$ of the $2\times 2$ 
polynomial system (in the variables $(x_1, x_2)$) below:
\begin{equation}
\begin{split}
m(x_1)&=0 \\
      g(x_1, x_2) &=0
\end{split} \label{eqnp4case2}
\end{equation}
which can be calculated in deterministic polynomial time. 
\end{thm}

The most interesting subcase is when
$s = 2$. 
From Equation~\ref{eqnp3}, we first build an ideal
\[ ( m(x_1), g(x_1, x_2)) \pmod{p} \subset \F_p[x_1, x_2].  \]
We can assume that the leading coefficient of 
$ g(x_1, x_2) $, viewed as a polynomial in $  x_2 $,
is invertible in $ \F_p[x_1]/(m(x_1))  $, thus the polynomial can be made monic.
If not, we can factor $ m(x_1) $, and use its factors as new $ m(x_1) $'s,
and restart the algorithm with $ m(x_1) $ of smaller degrees.  
So we may assume that the ideal is given as
\[ (m(x_1), x_2^{n_2} + f_2(x_1, x_2)),  \]
where $ n_2 \leq 2$ and the degree of $ x_2 $ in $ f_2 $ is less than $ n_2 $.
If $ (r, r_2) $ is a root in $ \F_p $ of the ideal, and $ r_1 $ is
the lift of $ r $ to the $ \Z_p $-root of $ m $,
then $ r_1 + p r_2 $ is a solution of $ f \pmod{p^3} $.
We compute the rational component of the ideal, and find its radical
over $ \F_p$. In the process, we may factor $ m(x) $ over $ \F_p $.
But how do we keep the information about $ p $-adic roots of $ m(x) $,
a polynomial with integer coefficients?

Our solution to this problem is to use Teichmuller lifting: 
Recall that for an element $\alpha$ in the prime finite field $\Z/p\Z$, the 
Teichmuller lifting of $\alpha$ is the unique $p$-adic integer 
$w(\alpha)\in \Z_p$ such that $w(\alpha)\equiv \alpha \mod p$ and 
$w(\alpha)^p=w(\alpha)$. If $a$ is any integer representative of $\alpha$, 
then the Teichmuller lifting of $\alpha$ can be computed by
$$w(\alpha) = \lim_{k\rightarrow \infty} a^{p^k}, \  w(\alpha) \equiv 
a^{p^t} \mod p^t.$$
Although the full Teichmuller lifting cannot be computed in finite time, 
we will see momentarily how its mod $p^t$ reduction can be computed in 
deterministic polynomial time. 

Let us now review how the mod $p^t$ reduction of the Teichmuller lift 
can be computed in deterministic polynomial time: If $m(x)\in \Z[x]$ is a 
monic polynomial of degree $d>0$ such that $m(x)\mod p$ splits as a product of 
distinct linear factors
$$m(x) \equiv \prod_{i=1}^d (x-\alpha_i) \mod p, \ \alpha_i \in \Z/p\Z, $$
then the Teichmuller lifting of $m(x)$ mod $p$ is defined to be the unique monic $p$-adic polynomial 
$\hat{m}(x) \in \Z_p[x]$ of degree $d$ such that the $p$-adic roots of $\hat{m}(x)$ are exactly the 
Teichmuller lifting of the roots of $m(x)\mod p$. That is, 
$$\hat{m}(x) =\prod_{i=1}^d (x  - w(\alpha_i)) \in \Z_p[x].$$
The Teichmuller lifting $\hat{m}(x)$ can be computed without factoring $m(x)\mod p$. Using the coefficients 
of $m(x)$, one forms a $d\times d$ companion matrix $M$ with integer entries such that  
$m(x) = \det(xI_d - M)$. Then, one can show that 
$$\hat{m}(x) = \lim_{k\rightarrow \infty} \det(xI_d -M^{p^k}), \  \hat{m}(x) \equiv \det(xI_d- M^{p^t}) \mod p^t.$$
This construction and computation of Teichmuller lifting of a single polynomial $m(x) \mod p$ 
can be extended to any triangular zero dimensional radical ideal with only rational roots as follows. 

Let $I$ be a radical ideal with only rational roots of the form 
$$I =(g_1(x_1), g_2(x_1, x_2), \cdots, g_k(x_1,\cdots, x_k))  \subset \F_p [x_1, x_2, \cdots, x_k],$$
where $g_i(x_1,\cdots, x_i)$ is a monic polynomial in $x_i$ of the form 
$$g_i(x_1,\cdots, x_i) = x_i^{n_i} + f_i (x_1, x_2, \cdots, x_i), \ n_i\geq 1$$
satisfying that the degree in $x_i$ of $f_i$ is less than $n_i$. Such a presentation of the ideal $I$ is 
called {\em triangular form}. It is clear that $I$ is a zero dimensional complete intersection. 
Using the companion matrix of a polynomial, we can easily find $n_i\times n_i$ matrices 
$M_{i-1}(x_1,..., x_{i-1}) $ whose entries are polynomials with coefficients 
in $\Z$ such that 
$$g_i(x_1,..., x_{i}) \equiv  \det(x_iI_{n_i}- M_i(x_1,..., x_{i-1})) \mod p, \ 1\leq i \leq k.$$
Recursively define the polynomial $f_i(x_1, \cdots, x_i) \in \Z/p^t\Z [x_1,\cdots, x_i]$ for $1\leq i\leq k$ 
such that 
$$f_1(x_1) \equiv \det (x_1I_{n_1} - M_0^{p^t}) \mod p^t,$$
$$f_2(x_1, x_2) \equiv \det (x_2I_{n_2} - M_1(x_1)^{p^t}) \mod (p^t, f_1(x_1)),$$
$$\cdots$$
$$f_k(x_1, \cdots, x_k) \equiv \det (x_kI_{n_k} - M_{k-1}(x_1,\cdots, x_{k-1})^{p^t}) \mod (p^t, f_1, \cdots, f_{k-1}).$$
The ideal $\hat{I} = (f_1, \cdots, f_k) \in \Z/p^t\Z [x_1,\cdots, x_i]$ is 
called the Teichmuller lifting mod $p^t$ of $I$. 
It is independent of the choice of the auxiliary integral matrices $M_i$.  
The roots of $\hat{I}$ over $\Z/p^t\Z$ are precisely the Teichmuller liftings  
mod $p^t$ of the roots of $I$ over $\F_p$. 
Each point $(r_1, \cdots, r_k)$ over $\Z/p^t\Z$ of $\hat{I}$ satisfies the condition $r_i^p \equiv r_i \mod p^t$. 

We require that $ m(x) $ be the Teichmuller lift at beginning of the algorithm.
Then we compute the Teichmuller lift of the ideal,
which is an ideal in $ \Z_p [x_1, x_2] $. We only need
it modulo $ p^4 $.
Denote the ideal by $ I_2 $.
For every root $ (r_1, r_2) $ of $ I_2 $, $ r_1 + pr_2 $ is
a solution of $ f(x) = 0 \pmod{p^3} $. Namely,
for any integer $ r_3 $,
we have $ f(r_1 + p r_2 + p^2 r_3) = 0 \pmod{p^3}. $

According to Theorem \ref{Tayg}, there exists a polynomial $ G(x_1, x_2, x_3) $ 
such that 
\[  f(x_1 + px_2  + p^2 x_3) \equiv p^3 G(x_1, x_2, x_3) \pmod{I_2},   \]
since $ I_2 \pmod{p} $ is radical. 
We have
 \begin{align*}
   f(x_1 + p x_2 + p^2 x_3) = g_{1} (x_1, x_2) p^3 x_3 + g_{0} (x_1, x_2) p^3 \pmod{(I_2, p^4)}. 
 \end{align*}	
 Hence if $ (r_1, r_2) $ is a root of $ I_2 $, then $ r_1 + pr_2 + p^2 r_3 $ is
 a root of $ f \pmod{p^4} $ iff $ (r_1, r_2, r_3) $ satisfies
 \[  g_{1} (x_1, x_2)  x_3 + g_{0} (x_1, x_2). \]
 Assume that $ g_1 \not = 0 $ and it does not vanish
 on any of the roots of $ I_2 \pmod{p} $. We count the rational roots
 of
 \[  (I_2, g_{1} (x_1, x_2)  x_3 + g_{0} (x_1, x_2))  \pmod{p}
  \subset
 \F_p [x_1, x_2, x_3].  \]  Multiplying the number by
 $ p $ gives us the number of $ \Z/(p^4) $ roots of $ f $.   

\section{Generalization to Arbitrary $t\geq 4$}
\label{sec:gen} 
We now generalize the idea for the case of $ t=4 $
to counting roots in $\Z/(p^t)$ of $ f(x) $ when $ t\geq 5 $ and 
$f$ is not identically $0$ mod $p$. (We can of course divide $f$ by 
$p$ and reduce $t$ by $1$ to apply our methods here, should $p|f$.)  
In the algorithm, we build a tree of ideals. At level $ k $, the ideals belong
to the ring $ \Z/(p^t) [x_1, x_2, \cdots, x_k] $.
The root of the tree (level 0) is $\{ 0\} \subset \Z/(p^t)$, the zero ideal.
At the next level the ideals are $ ( m(x_1)) $,
where $ m(x_1) $ is taken to be the Teichmuller lift of  $ f_i $ in 
Equation~\ref{modpfac}. 
We study how the roots in $ \Z_p $ of $ m(x_1) $ can be lifted to solutions
of $ f(x) $ in $ \Z/p^t $.

Let $ I_0, I_1, \cdots, I_k $ be the ideals in a path
from the root to a leaf. We require:
\begin{itemize}
\item $ I_0 = \{0\} \subset \Z/(p^t) $ and $ I_i \subset \Z/(p^t)[x_1, x_2, \cdots, x_i] $;  
\item $ I_i = I_{i+1} \cap \Z/(p^t)[x_1, x_2, \cdots, x_i] $ for all
  $ 0\leq i \leq k-1 $ ;
\item The ideal $ I_i \pmod{ p } $ is a zero dimensional and radical
  ideal with only rational roots in $ \F_p [x_1, x_2, \cdots, x_i] $ 
  for all $ 0\leq i\leq k $; Furthermore, $ I_i $ can be written as
  \begin{equation}
\begin{split}\label{idealform}
     &( I_{i-1},    x_i^{n_i} + f_i (x_1, x_2, \cdots, x_i))\\
     \subset &\Z/(p^t) [x_1, x_2, \cdots, x_i]
\end{split}
  \end{equation}
  where degree of $ x_i $ in $ f_i $ is less than $ n_i $.
\item The ideal $ I_i $ is (the mod-$ p^t $ part of) the  Teichmuller lift of $ I_i \pmod{p} $.
\end{itemize}

The basic strategy of the algorithm is to grow every branch of the tree until 
we reach a leaf that whose ideal allows a trivial count of the solutions. 
(In which case we output the count and terminate the branch.) 
If all branches terminate then we compute the summation of
the numbers on all the leaves as the output of the algorithm.
The tree of ideals contains complete information about the solutions of
$ f \pmod{p^t} $ in the following sense:
\begin{itemize}
\item For any ideal  $ I_i $ in the tree,
  there exists an integer $ s $, such that    $ i\leq s\leq t $,
and if $ (r_1, r_2, \cdots, r_i) $ is  a solution of $ I_i $ in
  $ ( \Z/(p^t))^i $,
  then $ r_1 + p r_2 + \cdots + p^{i-1} r_i  + p^i r $ is a solution of
  $ f(x) \pmod{p^s} $ for any integer $ r $. Denote the maximum such $s$ by $ s(I_i) $. 
\item If $ r $ is a root of $ f \pmod{p^t}$, then there exists a terminal leaf
  $ I_k $ in the tree such that 
  \[  r \equiv   r_1 + p r_2 + \cdots + p^{k-1} r_k  
  \pmod{p^k}   \]
  for some root $ (r_1, r_2, \cdots, r_k) $ of $ I_k $.
\end{itemize}

Suppose in the end of one branch we have
an ideal $ I_k \subset \Z/(p^t) [x_1, x_2, \cdots, x_k]$.
The ideal $ I_k \pmod{p} $
is zero dimensional and radical in $ \F_p $ with only rational roots.
There are two termination conditions:
\begin{itemize}
\item If $ s(I_k) \geq t $,   then each root of $ I_k $ in $ \Z_p^{k} $  
can produce $ p^{t-k} $ roots of $ f(x) $ in $ \Z/(p^t) $.
We can count the number of roots in $ \F_p^{k} $ of $ I_k $,
multiply it by $ p^{t-k} $, output the number, and terminate the branch.
\item Let $ g(x_1, x_2, \cdots, x_{k+1}) $ be the polynomial satisfying  
  \[  f(x_1+px_2 + p^2 x_3 + \cdots + p^{k-1} x_k + p^{k} x_{k+1}) \equiv
  p^{s(I_k)} g(x_1, x_2, \cdots, x_{k+1}) \pmod{I_k}.   \]
Such a polynomial exists according to Theorem \ref{Tayg}.
  Let $ D (x_1, x_2, \cdots, x_{k})$
  be the discriminant of $ g $, viewed as a polynomial in $ x_{k+1} $.
  Another termination condition is that none of the roots of $ I_k $
  vanishes on $ D $.  In this case, the count on this leaf is the number of
  rational roots of $ ( I_k, g ) \pmod{p} \subset \F_p [x_1, x_2, \cdots, x_{k+1}]. $ 
\end{itemize}

\begin{ex}
  If $ I_1 = (m(x_1)) $ where $ m(x_1) $ is the lift to $ \Z[x] $ of $ f_1 $ in Equation~\ref{modpfac},
  then $ s(I_1) = 1 $, and $ g(x_1, x_2) = x_2 (df/dx)(x_1) \pmod{p} $
  and $ gcd((df/dx) \pmod{p}, f \pmod{p}) = 1 $.
  So $ I_1 $ is a terminal leaf. \dia 
\end{ex}

If none of the conditions holds, 
let \[ g = \sum_{ j\leq t/k } g_{j}(x_1, x_2, \cdots, x_k) x_{k+1}^j.
  \]
  The degree bound $ t/k $ is due to the fact that  
  $ p^{kj} $ divides any term in the monomial expansion of
  $ f ( x_1 + p  x_2 + \cdots + p^{ k-1 }  x_k + p^{k} x_{k+1}) $ that has a
  factor $  x_{k+1}^j $. 
If any of the non-constant $g_{j}$ vanish at some rational
root of $ I_k $ in $  \F_p^{k}$ then this allows $ I_k \pmod{p} $ to decompose. 
Otherwise, for the ideal $  (I_k, g) \subset \Z/(p^t) [x_1, x_2, \cdots, 
x_{k+1}] $, we compute its  decomposition in  $ \F_p[x_1, x_2, \cdots, 
x_{k+1}] $ according to multiplicity type, 
find the radicals of the underlying ideals, and then lift them back to 
$ \Z/(p^t) [x_1, x_2, \cdots, x_{k+1}] $. 
They become the children of $ I_k $. 
Note that if $(I_k, g)$
does not have rational roots, it means that none of
the roots of $ I_k $ can be lifted to solution of $ f \pmod{p^{ s+1 }}$,
and thus the branch terminates with count $ 0 $.

\bigskip
\noindent 
{\bf Proof of Theorem~\ref{main}:} 
If $ p \leq d $ then factoring polynomials over $ \F_p $ can be done
  in time polynomial in $d$, and all the ideals in the tree are maximal.
  The number of children that $ I_k $ ($ k >1 $) can have is bounded from 
above by $ t/k $, the degree of $ g $. 
  (More precisely, number of nonterminal children is bounded from above by 
$ t/(2k) $.)
  The complexity is determined by the size of the tree, which is bounded 
from above by $ \prod_{1\leq k \leq t} (t/k) < e^t $.

  If $ p > d $ then we need to compute in the ring $ \F_p [x_1, x_2, \cdots,
  x_k]/I_k $. Observe that in~(\ref{idealform}), we must have $ n_i < t/(i-1) $
  for $ i\geq 2 $. So the ring is a linear space over $ \F_p $ with dimension
  $ \prod_{2\leq k \leq t} n_i < e^t $.   \qed 

\section{Computer Algebra Discussion}
In this section, we explain how to split ideals over $\F_p$ into triangular 
form so that the Teichmuller lift to $\Z_p$ can be computed. 
We start with the one variable case:  for any given ideal $I = (f(x)) \subset \F_p[x]$,  we can split $f(x)$ into the following form
\[ f(x) =  g_1(x)^{d_1} \cdots g_t(x)^{d_t}  g_0(x) \]
where $d_1> \cdots > d_t >0$,  the polynomials $g_1(x), \cdots, g_t(x) \in \F_p[x]$ are separable, pairwise co-prime  and each splits completely over $\F_p$, and 
$g_0(x)$ has no linear factors in $\F_p[x]$. This can be computed deterministically in time polynomial in $\log(p) \deg(f)$.  Note that, for $1 \leq i \leq t$,
each root of $g_i(x)$ has multiplicity $d_i$ in $I$. This  means that we can  count the number of $\F_p$-rational roots  of $I$ and their multiplicities  in polynomial time.  
Also,  the rational part of $I$ (i.e., excluding the part of $g_0(x)$) is decomposed into $t$ parts $g_1(x), \ldots, g_t(x)$.

Now we show how to go from $k$ variables to $k+1$ variables for any $k\geq 1$. Suppose $J =(g_1, g_2, \ldots, g_k) \subset \F_p[x_1, \ldots x_k]$ has a triangular form:
\begin{eqnarray*}
    g_1  & = &  x_1^{n_1} + r_1(x_1), \\
    g_2 & = &  x_2^{n_2} + r_2(x_1, x_2), \\
    & \vdots & \\
   g_k  & = &  x_k^{n_k} +  r_k(x_1, x_2, \ldots, x_k),
\end{eqnarray*}
where $g_i$ is monic in $x_i$ (i.e.,  the degree of $r_i$  in $x_i$  is less than  $<n_i$) for $1 \leq i \leq k$.
We further assume that $J$ is radical and completely splitting over $\F_p$, that is,   $J$ has $n_1 n_2 \cdots n_k$ distinct solutions in $\F_p^k$.  In particular,
$g_1(x_1)$ has $n_1$ distinct  roots in $\F_p$  and,  for each root $a_1 \in \F_p$ of $g_1(x_1) $, there are $n_2$ distinct $a_2 \in \F_2$ so that $(a_1, a_2)$ is a solution
of $g_2(x_1, x_2)$. In general, for  $1 \leq i < k$,  each solution $(a_1, \ldots, a_{i}) \in \F_p^i$ of $g_1, \ldots, g_i$ can be extended to $n_{i+1}$ distinct solutions
$(a_1, \ldots, a_{i}, a_{i+1}) \in \F_p^{i+1}$ of $g_{i+1}$.  For convenience, any ideal with these properties is called a {\em splitting triangular  ideal}.

Let $f \in \F_p[x_1, \ldots, x_k, x_{k+1}]$ be any nonzero polynomial which is monic in $x_{k+1}$, and let $I =(J, f)$ be the ideal generated by $J$ and $f$ in 
$\F_p[x_1, \ldots, x_k, x_{k+1}]$. We want to decompose $I$ into splitting triangular ideals, together with their multiplicities. More precisely, we want to decompose $I$  into the following form:
\begin{equation} \label{G1}
I = (J_1,  h_1^{d_1})  \cap (J_2,  h_2^{d_2})  \cdots \cap (J_m,  h_m^{d_m})  \cap (J_0, h_0),
 \end{equation}
where $J = J_1 \cap J_2 \cap \cdots \cap J_m\cap J_0$,    $I_0= (J_0, h_0)$ has no solutions in $\F_p^{k+1}$, and 
the ideals $I_i = (J_i,  h_i)\subset \F_p[x_1, \ldots, x_k, x_{k+1}]$, $1 \leq i \leq m$, are  
splitting triangular ideals and  are pairwise co-prime (hence any distinct 
pair of them have no common solutions).

To get the decomposition (\ref{G1}),  we first compute
\[        w := x_{k+1}^p -x_{k+1} \bmod G. \]
where  $G=\{g_1, g_2, \ldots, g_k, f\}$ is a Gr\"{o}bner basis under the lexicographical order with $x_{k+1} > x_k > \cdots > x_1$. 
Via the square-and-multiply method,  $w$ can be computed using O$(\log(p)^3 n^2)$ bit operations where $n= \deg(f) \cdot n_1 \cdots n_k$ is the degree
of the ideal $I$.   Next we compute the Gr\"{o}bner basis  $B$ of  $\{g_1, g_2, \ldots, g_k, f, w\}$ (under lex order with $x_{k+1} > x_k > \cdots > x_1$), 
which is radical and completely splitting (hence all of its solutions
are in $\F_p^{k+1}$ and are distinct).   This mean that we get rid of the nonlinear part $(J_0, h_0)$ in (\ref{G1}).  The ideal $(B)$ is now equal to the radical of
 the rational part of $I$.   To decompose $(B)$ into splitting triangular ideals, we view each polynomial in $B$ as a polynomial in $x_{k+1}$ with coefficient in 
$\F_p[x_1, \ldots, x_k]$. Let $t_0=0 < t_1 < \cdots < t_v$ be the distinct degrees  of $x_{k+1}$ among the polynomials in $B$. 
For $0 \leq i \leq v$, let $B_i$ denotes the set of the leading coefficient  of all $g \in B$ with $\deg(g) \leq  t_i$.   We have the chain of ideals
\[ J \subseteq (B_0) \subset  (B_1) \subset \cdots \subset (B_{v-1}) \subset (B_v) = \F_p[x_1, \ldots, x_k],\]
with the following properties:
\begin{itemize}
\item[(i)]   $1 \in B_v$,
\item[(ii)]  each $B_i$ ($1 \leq i \leq v$) is automatically a Gr\"{o}bner basis under the lex order with $x_k > \cdots > x_1$ (one can remove some redundant polynomials from $B_i$),
\item[(iii)]   for $0 \leq i <v$, each solution of $B_{i}$ that is not a solution of $B_{i+1}$ can be extended to exactly $t_{i+1}$ distinct solutions of $I$. 
\end{itemize}

We can compute a Gr\"{o}bner  basis  $C_i$ for  the colon ideal   $(B_{i+1}):(B_i)$ for $0 \leq i < v$. These $C_i$'s gives us the different components of $J$ that have
different number of solution extensions. Together with $B$,  we get different components of $(I, w)$. These components are completely splitting, but may not be in  triangular form (as stated above). We again use Gr\"{o}bner basis structure to further decompose them until  all are splitting triangular ideals $(J_i, h_i)$.  Note that computing Gr\"{o}bner bases is 
generally NP-hard. However, all of our ideals are of a special form, and 
their Gr\"{o}bner bases can be computed deterministically in polynomial time 
via the incremental method in \cite{GGV} (see also \cite{GVW}).

Finally, to get the multiplicity of each component $(J_i, h_i)$, we compute 
the Gr\"{o}bner basis for  the ideal $(J_i, f, f^{(j)})$ where $f^{(j)}$ 
denotes the $j$-th derivative of $f$
for $j=1, 2, \ldots, \deg(f)$, until the Gr\"{o}bner basis is $1$. These ideals may not be in triangular form, so may split further. But the total number of components is at most the degree of $f$. Hence the total number of bit operations used  is still polynomial in $\log(p) \deg(I)$.

\bibliographystyle{amsalpha}

\begin{thebibliography}{A}

\bibitem{cantorqp} David G.\ Cantor and Daniel M.\ Gordon,
{\it ``Factoring polynomials over $p$-adic fields,''}
Algorithmic number theory (Leiden, 2000), pp.\ 185--208,
Lecture Notes in Comput.\ Sci., 1838, Springer, Berlin, 2000.

\bibitem{denefver} Wouter Castryck; Jan Denef; and Frederik 
Vercauteren, {\it ``Computing Zeta Functions of Nondegenerate Curves,''}
International Mathematics Research Papers, vol.\ 2006, article ID
72017, 2006.

\bibitem{chambert} Antoine Chambert-Loir, {\it ``Compter (rapidement)
le nombre de solutions d'\'equations dans les corps finis,''}
S\'eminaire Bourbaki, Vol.\ 2006/2007, Ast\'erisque No.\ 317 (2008),
Exp.\ No.\ 968, vii, pp.\ 39-–90.

\bibitem{chistov} Alexander L.\ Chistov, {\it ``Efficient
Factoring [of] Polynomials over Local Fields and its Applications,''}
in I.\ Satake, editor, Proc.\ 1990 International Congress of Mathematicians,
pp.\ 1509--1519, Springer-Verlag, 1991.

\bibitem{cohenant} Henri Cohen, {\it A course in computational
algebraic number theory,} Graduate Texts in Mathematics, 138, Springer-Verlag, 
Berlin, 1993. 

\bibitem{denef} Jan Denef, {\it ``Report on Igusa's local zeta function,''} 
S\'eminaire Bourbaki 1990/1991 (730-744) in Ast\'erisque 
201--203 (1991), pp.\ 359--386.

\bibitem{gouvea} Fernando Q.\ Gouve\^ea, {\it $p$-adic
Numbers,} Universitext, 2nd ed., Springer-Verlag, 2003.  


\bibitem{Gao01}
Shuhong Gao,
{\it ``On the deterministic complexity of polynomial factoring''}, 
{\em Journal of Symbolic Computation},  31 (2001), 19--36.

\bibitem{GGV}
Shuhong Gao, Yinhua Guan and Frank Volny IV,
{\it ``A new  incremental algorithm for computing Gr\"{o}bner bases''},
the 35th International Symposium on Symbolic and Algebraic Computation (ISSAC), pp.\ 13--19, 
Munich, July 25--28, 2010.

\bibitem{GVW}
Shuhong Gao, Frank Volny IV and Mingsheng Wang,
{\it ``A new  framework  for computing Gr\"{o}bner bases''},  
 {\em Mathematics of Computation},  85 (2016),  no.\ 297,   449--465.


\bibitem{gnp} Jordi Gu\`ardia; Enric Nart; Sebastian Pauli, 
{\it ``Single-factor lifting and factorization of polynomials over 
local fields,''} Journal of Symbolic Computation 47 (2012), pp.\ 1318--1346. 

\bibitem{hjpw} Trajan Hammonds; Jeremy Johnson; Angela Patini;
and Robert M.\ Walker, {\it ``Counting Roots of Polynomials Over $\Z/p^2\Z$,''}
Math ArXiv preprint {\tt 1708.04713} .

\bibitem{igusa} Jun-Ichi Igusa, {\it Complex powers and asymptotic expansions 
I: Functions of certain types}, Journal f\"{u}r die reine und 
angewandte Mathematik, 1974 (268--269): 110–130.   

\bibitem{kedlaya} Kiran Kedlaya and Christopher Umans, {\it ``Fast polynomial 
factorization and modular composition,''} 
SIAM J.\ Comput.,  40 (2011), no.\ 6, pp.\ 1767--1802.  

\bibitem{lauder} Alan G.\ B. Lauder,
{\it ``Counting solutions to equations in many variables over finite
fields,''} Found.\ Comput.\ Math.\ 4 (2004), no.\ 3, pp.\ 221--267.

\bibitem{lauderwan} Alan G.\ B.\ Lauder and Daqing Wan, {\it
``Counting points on varieties over finite fields of small characteristic,''}
Algorithmic number theory: lattices, number fields, curves and cryptography,
pp.\ 579-–612, Math.\ Sci.\ Res.\ Inst.\ Publ., 44, Cambridge Univ.\ Press,
Cambridge, 2008.

\bibitem{lll} Arjen K.\ Lenstra; Hendrik W.\ Lenstra (Jr.);
Laszlo Lov\'asz, {\it ``Factoring polynomials with rational coefficients,''}
Math.\ Ann.\ 261 (1982), no.\ 4, pp.\ 515--534.

\bibitem{mw} Michael Maller and Jennifer Whitehead, 
{\it ``Efficient $p$-adic cell decomposition for univariate
polynomials,"} J.\ Complexity {\em 15} (1999), pp.\ 513-525.

\bibitem{bjornbm} Bjorn Poonen, {\it ``Heuristics
for the Brauer-Manin Obstruction for Curves,''} Experimental Mathematics,
Volume 15, Issue 4 (2006), pp.\ 415--420.

\bibitem{dwan} Daqing Wan, {\it ``Algorithmic theory of zeta functions over finite fields,''}
Algorithmic number theory: lattices, number fields, curves and cryptography,
pp.\ 5551-578, Math.\ Sci.\ Res.\ Inst.\ Publ., 44, Cambridge Univ.\ Press,
Cambridge, 2008.

\bibitem{zg} W.\ A.\ Zuniga-Galindo, {\it ``Computing Igusa's Local
Zeta Functions of Univariate Polynomials, and Linear Feedback Shift
Registers,''} Journal of Integer Sequences, Vol.\ 6 (2003), Article 03.3.6.

\end{thebibliography}

\end{document}